\documentclass[10pt]{article}
\usepackage{a4}
\usepackage[utf8]{inputenc}
\usepackage{latexsym}
\usepackage{amssymb,amsthm,amsmath, enumitem}
\usepackage{dsfont}
\usepackage[flushleft]{paralist}
\usepackage[all]{xy} 
\usepackage{hyperref} 
\usepackage{todonotes} 
\usepackage[OT2,T1]{fontenc}
\usepackage{stmaryrd}

\DeclareSymbolFont{cyrletters}{OT2}{wncyr}{m}{n}
\DeclareMathSymbol{\Sha}{\mathalpha}{cyrletters}{"58}

\makeatletter
\newcommand{\subjclass}[2][2010]{%
  \let\@oldtitle\@title%
  \gdef\@title{\@oldtitle\footnotetext{#1 \emph{Mathematics subject classification.} #2}}%
}
\newcommand{\keywords}[1]{%
  \let\@@oldtitle\@title%
  \gdef\@title{\@@oldtitle\footnotetext{\emph{Key words and phrases.} #1.}}%
}
\makeatother

\newcommand{\Z}{\mathds{Z}}
\newcommand{\N}{\mathds{N}}
\newcommand{\Q}{\mathds{Q}}

\newcommand{\C}{\mathds{C}}

\newcommand{\p}{\mathfrak{p}}

\newcommand{\Gal}{\textup{Gal}}
\newcommand{\rg}{\textup{rank}}

\newcommand{\Ok}{\mathcal{O}}

\newcommand{\Cl}{\mathrm{Cl}}

\newcommand{\frp}{\mathfrak{p}}
\renewcommand{\frm}{\mathfrak{m}}

\newcommand{\Ir}{\mathrm{Ir}}
\newcommand{\ord}{\mathrm{ord}}

\title{On the Gross order of vanishing conjecture\\ for large vanishing orders}

\author{Martin Hofer and Sören Kleine}
\date{}

\newtheorem{lemma}{Lemma}[section] 
\newtheorem{prop}[lemma]{Proposition} 
 
\newtheorem{conjecture}[lemma]{Conjecture} 
\newtheorem{thm}[lemma]{Theorem} 
\newtheorem*{thm*}{Theorem} 
\newtheorem{cor}[lemma]{Corollary}
\newtheorem{claim}[lemma]{Claim}

\theoremstyle{plain} 
\newcommand{\thistheoremname}{}
 \newtheorem*{genericthm*}{\thistheoremname}
\newenvironment{namedthm*}[1]
  {\renewcommand{\thistheoremname}{#1}%
   \begin{genericthm*}}
  {\end{genericthm*}}

\theoremstyle{definition}

\newtheorem{rem}[lemma]{Remark} 
 
\newtheorem{example}[lemma]{Example}

\DeclareMathOperator{\Hom}{\mathrm{Hom}}
\DeclareMathOperator{\GVC}{\mathrm{GVC}}
\DeclareMathOperator{\GKC}{\mathrm{GKC}}

\begin{document}

\maketitle

\begin{abstract}
 We prove the Gross order of vanishing conjecture in special cases where the vanishing order of the character in question can be arbitrarily large. In almost all previously known cases the vanishing order is zero or one. 
 
 One major ingredient of our proofs is the equivalence of this conjecture to the Gross-Kuz'min conjecture. We present here a direct proof of this equivalence, using only the known validity of the Iwasawa Main Conjecture over totally real fields. 
\end{abstract}

\section{Introduction} \label{introduction} 
One of the major breakthroughs in the theory of $L$-functions of number fields in recent years was the proof of one part of the Gross-Stark conjecture for one-dimensional characters by Dasgupta, Kakde and Ventullo (\cite{DaKaVe18}). This conjecture compares, in an appropriate setting, the leading terms of the Deligne-Ribet $p$-adic $L$-function and the Artin $L$-function at $s=0$ for each character. Although the main result of \cite{DaKaVe18} is unconditional, only the \emph{Gross order of vanishing conjecture} (abbreviated by $\GVC$) asserts that the vanishing orders of the two $L$-functions are equal.

Currently, $\GVC$ is basically only known unconditionally in cases where the vanishing order of the Artin $L$-function at $s=0$ is zero or one, e.g if the base extension $K / \Q$ is abelian, by work of Gross \cite{Gro81} or Emsalem \cite{Ems83}.

Seemingly completely independent of the above, there is a classical conjecture in Iwasawa theory, called \emph{Gross-Kuz'min conjecture} (abbreviated by $\GKC$), which asserts that for the cyclotomic $\Z_{p}$-extension of a number field $K$ ($p$ a fixed prime) with Galois group $\Gamma$, the quotient of $\Gamma$-coinvariants of a certain Iwasawa module is finite. This conjecture (in this generality) has been formulated first by Kuz'min in \cite{Kuz72} and independently by Jaulent in \cite{Jau86}. It is known for abelian extensions $K / \Q$ by work of Greenberg (\cite{green-l-adic}) and it has been studied intensely by Jaulent, Kolster, Kuz'min, Nguyen Quang Do and many others. 

In the first part of this manuscript we show that $\GVC$ is (under a standard assumption) equivalent to an appropriate characterwise variant of $\GKC$. In particular, we allow the characters to have arbitrary dimension and the base extensions to be non-abelian over $\Q$.

In the second part we construct examples with large vanishing orders for which we prove that $\GKC$ and therefore, via the equivalence, also $\GVC$ holds.

In the remaining part of the introduction we state our results more explicitly and sketch our general proof strategy.

\paragraph{The Gross-Kuz'min conjecture}
Fix a prime $p \neq 2$. Let $K$ be a number field and let $K_{\infty}/K$ be the cyclotomic $\Z_{p}$-extension with Galois group $\Gamma$. We denote by $K_{n}$ the $n$-th layer of this extension. Furthermore, let $A_{K_{n}}$ be the $p$-Sylow subgroup of the ideal class group of $K_{n}$ and $B_{K_{n}}$ the subgroup of $A_{K_{n}}$ generated by the prime ideals above $p$. We set $A'_{K_{n}}=A_{K_{n}}/B_{K_{n}}$ and $A_{K_{\infty}}^{\prime}:= \varprojlim_{n} A'_{K_{n}}$, where the projective limit is taken with respect to norm maps. With this notation the Gross-Kuz'min conjecture can be formulated as the assertion that the $\Gamma$-coinvariants of $A_{K_{\infty}}^{\prime}$ are finite.

Let now $K$ be a finite Galois CM-extension of a totally real number field $R$ with Galois group $G$, and  let $\chi$ be a totally odd irreducible character of $G$ over $\C$. 

In the following, we assume that $K \cap R_\infty = R$, where $R_\infty$ denotes the cyclotomic $\Z_p$-extension of $R$. For the character $\chi$,  $V_{A'}^{(\chi)}$ shall denote the $\chi$-component of the vector space $V_{A'}:= A_{K_{\infty}}^{\prime} \otimes_{\Z_p} \overline{\Q}_{p}$ and $f_{A',\chi}(T) \in \overline{\Q}_{p}[T]$ shall be the characteristic polynomial of the linear map induced by the action of $T = \gamma - 1$ on $V_{A'}^{(\chi)}$, where $\gamma \in \Gamma$ denotes a topological generator. With this notation we can formulate the following conjecture: 
\begin{namedthm*}{Conjecture~$\GKC(K/R, \chi)$}
\label{conj_GKC_comp_intro}
Let $K$ be a Galois CM-extension of a totally real field $R$ satisfying $K \cap R_\infty = R$.
If $\chi$ is a totally odd irreducible character of $\Gal(K/R)$ over $\C$, then $f_{A',\chi}(T)$ is not divisible by $T$.
\end{namedthm*}

\paragraph{The Gross order of vanishing conjecture}

We keep the notation from the last paragraph and additionally introduce a set $S$ which consists of the places above $p$ and above $\infty$. Then for an irreducible character of $G$ over $\C$, we denote by $L_{S}(s, \chi)$ the $S$-truncated Artin $L$-function. Moreover, for a totally even character $\psi$ of $G$ over $\C_{p}$ we denote by $L_{p, S}(s, \psi)$ the $S$-truncated Deligne-Ribet $p$-adic $L$-function (see e.g. \cite{DeRi80} and \cite{Gre83}). Then we can recall the following conjecture.
\begin{namedthm*}{Conjecture~$\GVC(K/R, \chi)$}
\label{conj_GVC_intro}
Let $K$ be a Galois CM-extension of a totally real field $R$ with Galois group $G$. If $\chi$ is a totally odd irreducible character of $G$ over $\C_{p}$, then
\[
\mathrm{ord}_{s=0}L_{p, S}(s, \check \chi \omega_{R}) = \mathrm{ord}_{s=0} L_{S}(s, \chi); 
\]
here we identify $\chi$ with a totally odd irreducible character of $G$ over $\C$ via a fixed isomorphism $j\colon \C \cong \C_{p}$. Moreover, $\check \chi$ is the contragredient character and $\omega_{R}$ is the Teichmüller character.
\end{namedthm*}
In \cite{Bur18} the inequality '$\geq$' of this conjecture is shown unconditionally. We will recover this inequality (under our hypothesis $K \cap R_\infty = R$) in course of the proof that the two conjectures are equivalent in an appropriate setting.

\paragraph{Equivalence of the conjectures}

Fix $K/R$ and $G =\Gal(K/R)$ as above, and recall that $R_{\infty}$ denotes the cyclotomic $\Z_{p}$-extension of $R$.

In Section \ref{section:equivalence}, we will prove the following 
\begin{namedthm*}{Theorem~A}
\label{thm_equiv_intro}
Let $K$ be a finite Galois CM-extension of a totally real number field $R$ with Galois group $G$, let $p$ be an odd prime and let $\chi$ be a totally odd irreducible character of $G$ over $\C$ resp. $\C_{p}$. Assume that $K \cap R_{\infty} = R$. Then $\GKC(K/R,\chi)$ holds if and only if $\GVC(K/R, \chi)$ holds.
\end{namedthm*} 

Actually an equivalence of appropriate formulations of the Gross-Kuz'min and the Gross order of vanishing conjectures can be obtained more generally, namely without the assumption $K \cap R_{\infty} = R$, by combining recent and deep results of Burns \cite{Bur18} with work of Kolster \cite{Kol91} (cf. Remark~\ref{rem:final}). We give here a more direct proof, the main ingredients of which are the Iwasawa main conjecture over totally real fields proved by Wiles (\cite{Wil90}), Brauer Induction and a careful analysis of the module $(B_{H_{\infty}})^{-}$ for  $H_\infty= R_\infty \cdot H$ for a suitable finite extension $H$ of $R$.

\paragraph{Examples with large vanishing order}
To the knowledge of the authors in almost all proven cases of $\GVC(K/R, \chi)$ the vanishing order of $L_{S, \chi}(s, \chi)$ at $s=0$ is $0$ or $1$, e.g. if $R=\Q$.

In view of Theorem~A, it should not be a surprise that also (almost, cf.~Remark~\ref{rem_GKC_known}) all known instances of the Gross-Kuz'min conjecture are cases where $K$ contains only one or two primes above $p$ or where $K / \Q$ is abelian. 
So evidence for these important conjectures for 'higher order of vanishing' is scarce. This is, in part, based on the fact that most of the proofs depend on the famous 'transcendence result' of Baker/Brumer, the necessary generalization of which is only conjectural (for work on the $p$-adic Schanuel conjecture in connection to the above conjectures see \cite{Kuz18} and \cite{Ems83}). 

We also use the Theorem of Brumer/Baker in our results. First, we derive the following result:
\begin{namedthm*}{Theorem~B}
For infinitely many primes and infinitely many non-abelian extensions $K/R$ of degree $8$, ${\GVC(K/R, \chi)}$ holds for all characters ${\chi \in \Ir^-(\Gal(K/R))}$. This includes examples with arbitrarily large vanishing order.
\end{namedthm*}
The main input here is the known validity of Leopoldt's conjecture and the equivalence to the Gross-Kuz'min conjecture in this specific setting.

In \cite{green-l-adic}, Greenberg used the Theorem of Brumer/Baker in order to prove the minus part of $\GKC$ for a CM-field $K$, provided that $K/\Q$ is abelian and that the prime $p$ is totally split in $K$ (moreover, he showed how to circumvent the latter condition in the abelian case). Using the approach of \cite{green-l-adic} and building on results from \cite{T-ranks} about the asymptotic growth of Galois coinvariants in a $\Z_p$-extension, we can more generally use Greenberg's approach in order to derive from the Theorem of Brumer/Baker the following 
 \begin{namedthm*}{Theorem~C}
  Let $K$ be a finite abelian CM-extension of a totally real number field $R$, and suppose that there exists some prime $\p$ of $R$ above $p$ which is totally split in $K$, and such that $R_\p = \Q_p$. 
 
 Then 
  	$$ \rg_{\Z_p}((A_{K_{\infty}}')^-_\Gamma) \le r - s, $$ 
  	where $r$ denotes the number of primes of the maximal totally real subfield $K^+$ of $K$ which split in $K$, and $[K:R] = 2s$. 
\end{namedthm*} 
We then proceed by explaining how one can derive from Theorem~C examples with large vanishing orders (see Corollary~\ref{cor_GVC} and Example~\ref{ex_largevanish}) where $\GVC(K/R, \chi)$ holds; the main idea is to enlarge the totally real field $R \subseteq K$.

In a subsequent paper the authors plan to describe how one can use this work in order to verify both conjectures numerically in many interesting new cases.

We want to thank Dominik Bullach for suggesting to think about this subject in the first place and many helpful conversations at the initial phase of this project. We are very grateful to Pascal Stucky for 
many illuminating discussions, in particular with the first author. Moreover, we want to thank Cornelius Greither for several insightful comments and suggestions.

\subsection{Notation} \label{section:notation} 

Unless stated otherwise, $p$ will always denote a rational prime. 

We denote by $\overline{K}$ a fixed algebraic closure of a field $K$ with characteristic $0$ and by $G_{K}$ the absolute Galois group. 

For a number field $K$ we denote by $S_{\infty}(K)$ the set of infinite places in $K$ and by $S_{p}(K)$ the set of places above $p$ in $K$. 
Let $S$ be a finite set of places of $K$ containing $S_{\infty}(K)$. Then we denote by $E^{S}_{K}$ the group of $S$-units of $K$ and by $\Cl_{S}(K)$ the $S$-class group of $K$. 
Moreover, we write $\Cl(K)=\Cl_{S_{\infty}(K)}(K)$ and $E_{K}^{\prime} = E_{K}^{S_p(K)}$ (i.e., the group of $p$-units of $K$) for brevity.  

Let $K/R$ be a finite Galois extension with Galois group $G$. We denote by $\Ir(G)$ resp. $\Ir_{p}(G)$ the set of irreducible characters of $G$ over $\C$ resp. over $\C_{p}$ and for each such character $\chi$ we fix a representation $V_{\chi}$ over $\C$ resp. over $\C_{p}$. For each such $\chi$ we write $\breve{\chi}$ for its contragredient.  For each $\chi \in \Ir_{p}(G)$, we denote by $\varepsilon_\chi = \frac{\chi(1)}{|G|} \cdot \sum_{\sigma \in G} \chi(\sigma) \sigma^{-1} $ the corresponding idempotent.

Now suppose that $K$ is a CM-field, and that $R$ is totally real. Then we write $K^{+}$ for the maximal real subfield of $K$ and $\tau$ for the (unique) non-trivial element of $\Gal(K/K^{+})$. We say that $\chi \in \Ir(G)$ is even if ${\chi(\tau)=  \chi(1)}$, and odd if $\chi(\tau)= - \chi(1)$. 

More generally, suppose that $\chi$ denotes a character of the absolute Galois group $G_R$ of $R$ with finite order. Letting $\ker(\chi) = \{g \in G_R \mid \chi(g) = \chi(1)\}$, we let $R_\chi \subseteq \overline{R}$ be the subfield fixed by $\ker(\chi)$. 
Then $\chi$ is called totally even if $R_\chi$ is totally real, and it is called totally odd if $R_\chi$ is CM. 

In the special setting of characters of $G = \Gal(K/R)$, $K$ a CM-Galois extension of $R$, the notions odd and totally odd (respectively, even and totally even) coincide, because $R_\chi \subseteq K$ for each character $\chi$ of $G$. 

We write $\Ir^{\pm}(G)$ and $\Ir_{p}^{\pm}(G)$ for the subsets of $\Ir(G)$ and $\Ir_{p}(G)$ that consist of the characters which are totally even or totally odd, respectively. For any $G$-module $M$ we define $M^{\pm} := \{ m \in M : \tau(m): \pm m\}$; these are two $G$-submodules of $M$. 

We denote by $K_{\infty}/K$ the cyclotomic $\Z_{p}$-extension of a number field $K$ and by $K_{n}$ the unique intermediate field of degree $p^n$ over $K$, $n \in \mathbb{N}_{0}$.

\section{Two conjectures named after Gross}

In this chapter we formulate the Gross-Kuz'min conjecture, the  Gross order of vanishing conjecture and characterwise variants of both. Moreover, we recall the cases where the conjectures are unconditionally known.

\subsection{The Gross-Kuz'min conjecture} \label{section:gross_kuzmin} 
Fix a prime $p$ and a number field $K$. For each $n \in \N$, let $A_n' = A_n/B_n$ be the quotient of $A_n = \textup{Cl}(K_n)_p$, the $p$-Sylow subgroup of $\Cl(K_{n})$, by the subgroup $B_n$ generated by the primes of $K_n$ dividing $p$, and let $A := \varprojlim A_n$ and $A' := \varprojlim A_n'$, where the projective limits are taken with respect to the norm maps.

\begin{conjecture}[Gross-Kuz'min conjecture] \label{conj:GKC} $\;$ \\ Let $K_\infty/K$ be the cyclotomic $\Z_p$-extension of a number field $K$. Then the module of coinvariants $\left(A'\right)_\Gamma$, where $\Gamma = \Gal(K_\infty/K)$, is finite.
\end{conjecture}

In what follows, we write $\GKC(K)$ for ``the Gross-Kuz'min conjecture for the cyclotomic $\Z_p$-extension $K_\infty/K$''.
\begin{rem} \label{rem:gross-kuzmin} 
    Let $L/K$ be a finite extension of number fields, let $K_\infty/K$ be the cyclotomic $\Z_p$-extension, and let $L_\infty = L \cdot K_\infty$.  If $\GKC(L)$ holds, then also $\GKC(K)$ is true. 
\end{rem} 
 	
\paragraph{The minus Gross-Kuz'min conjecture} 
If $K$ is a CM-field and $p \ne 2$, then we consider minus parts ${(A')^- = \varprojlim_n (A_n')^-}$. The following conjecture is then a part of Conjecture~\ref{conj:GKC}: 

\begin{conjecture}[Minus Gross-Kuz'min conjecture] \label{conj:GKC-minus} $\;$ \\ 
  With the above notation, $((A')^-)_\Gamma$ is finite. 
\end{conjecture} 
We write $\GKC^-(K)$ for ``the minus Gross-Kuz'min conjecture for the cyclotomic $\Z_p$-extension $K_\infty/K$''. 

Note: Conjecture~\ref{conj:GKC-minus} is true if and only if 
\[ ((A')^-)_\Gamma \otimes_{\Z_p} \overline{\Q}_p = \{0\}. \] 
One can further reformulate this conjecture: 
$V_{A'}^- := (A')^- \otimes_{\Z_p} \overline{\Q}_p$ is a finitely generated $\overline{\Q}_p$-vector space. 
Let ${\gamma}$ be a topological generator of $\Gamma$. If ${f_{A'}^-(T) \in \overline{\Q}_p[T]}$ denotes the characteristic polynomial of the linear map induced by the action of $T = \gamma - 1$ on $V_{A'}^-$, then Conjecture \ref{conj:GKC-minus} holds if and only if $f_{A'}^-(T)$ is not divisible by $T$, i.e. it has nontrivial constant coefficient. For later use, we also introduce the characteristic polynomial $f_A^-(T)$ of $T$ on the vector space $V_A^- = A^- \otimes_{\Z_p} \overline{\Q}_p$, where $A^- = \varprojlim A_n
^-$. 

\paragraph{Character-wise minus Gross-Kuz'min conjecture} 
Let $K$ be a CM-field, ${p \ne 2}$, and let $R \subseteq K$ be a totally real subfield such that $K/R$ is normal. We write $G = \Gal(K/R)$. Recall the notation $R_\infty$ from the Introduction. In the following, we assume that $K \cap R_\infty = R$. 

Let $V_{A'} = A' \otimes_{\Z_p} \overline{\Q}_p$, $\chi \in \Ir(G)$,  $V_{A'}^{(\chi)} := \varepsilon_\chi \cdot V_{A'}$, and let ${f_{A',\chi}(T) \in \overline{\Q}_p[T]}$ be the characteristic polynomial of $T = \gamma - 1$ on $V_{A'}^{(\chi)}$. 
\begin{conjecture}
\label{conj:GKC-chi} 
  The characteristic polynomial $f_{A',\chi}(T)$ is not divisible by $T$. 
\end{conjecture} 
In the following, we will write $\GKC(K/R, \chi)$ for ``the characteristic polynomial $f_{A',\chi}(T)$ is not divisible by $T$''.

\begin{rem} \label{rem:chi-components} 
  Since $(A')^- \otimes_{\Z_p} \overline{\Q}_p = \bigoplus_{\chi \in \Ir^-(G)} V_{A'}^{(\chi)}$, Conjecture~\ref{conj:GKC-minus} is equivalent to the validity of Conjecture~\ref{conj:GKC-chi} for all $\chi \in \Ir^-(G)$. 
\end{rem}

The next comment will explain why Gross's name appears in the conjecture.

\begin{rem} \label{rem_kolster} $\;$ \\ \vspace{-4mm} 
  \begin{compactenum}[(1)] 
  \item There exists a different formulation of Conjecture~\ref{conj:GKC} which dates back to Jaulent's thesis \cite{Jau86}. It asserts that for a number field $K$ a certain invariant, called ``Gross-defect'' $\delta_{K}$, is zero. Indeed, \cite[Theorem~1.14]{Kol91} shows that $\delta_{K}=0$ if and only if $(A_n')^{\Gamma}$ is bounded independently of $n$, which in turn holds if and only if $(A')_{\Gamma}$ is finite.
  \item B. Gross's name appears in the conjecture because Conjecture~1.15 in \cite{Gro81} is equivalent to the Gross defect of \cite{Kol91} being zero if one considers CM-fields $K$ and minus-parts of the objects involved.
 \end{compactenum}
 \end{rem}

We now give a list of all (unconditional) results on the Gross-Kuz'min conjecture holds (as far as the authors are aware).
\begin{rem} $\;$ \\ \vspace{-4mm} \begin{compactenum}[(1)] 
\label{rem_GKC_known} 
\item Let $K$ be any number field. If $K$ contains exactly one prime above $p$, then $\GKC(K)$ follows from Chevalley's Theorem (see \cite[Lemma 4.1 in Chapter 13]{lang}); if $K$ contains two such primes then $\GKC(K)$ is known due to \cite[Theorem 3.4]{T-ranks}.
\item Let $K$ be a number field such that $K$ is normal over $\Q$, $\Gal(K/\Q) \cong S_{4}$, K contains an imaginary quadratic subfield, and $p$ is such that the decomposition group of some prime of $K$ dividing $p$ is isomorphic to a 3-Sylow subgroup of $\Gal(K / \Q)$. Then \cite[Corollary 4]{Kuz18} shows that $\GKC(K)$ holds.
\item Let $K/\Q$ be a normal extension, $G := \Gal(K/\Q)$, and let $p$ be a prime number such that the decomposition field $F_p = K^{D_p}$ of $p$ is normal over $\Q$ and totally real, and such that in the Artin-Wedderburn decomposition of the group ring $\Q_p[G/D_p]$ only matrix rings $M_n(R)$ with $n \le 2$ occur. Then $\GKC(K)$ holds for this prime $p$ (\cite{Jau02}). 
\item Greenberg proves $\GKC(K)$ for any abelian extension $K$ of $\Q$ in \cite{green-l-adic}.
\item Let $K$ be a CM-field. Then $\GKC^-(K)$ holds if $K$ contains a CM-subfield $k$ such that the prime $p$ does not split in $K/k$ and $\GKC^-(k)$ holds (see \cite[Section~3]{green-l-adic} and \cite[Scolie~10]{Jau17}). 
\end{compactenum}
\end{rem}

\subsection{The order of vanishing conjecture of Gross} \label{section:Gross_vanishing}
\label{subsec_GVC}
Now we describe the order of vanishing conjecture originally formulated as Conjecture~2.12 a) in \cite{Gro81}; we use the notation of \cite{Bur18}.

Let $K$ be a finite CM Galois extension of a totally real number field $R$. We fix a finite set $S$ of places in $R$ which contains $S_{\infty}(R)$ and $S_{p}(R)$. Let $\chi \in \Ir(G)$. Then we denote by $L_{S}(s, \chi)$ the $S$-truncated Artin $L$-function of $\chi$. By \cite[Proposition~3.4]{Tat84} we have
\begin{align}
\label{eqn_order_ArtinLseries}
r_{S, \chi}:=\ord_{s=0} L_{S}(s, \chi)= \dim_{\C}(\Hom_{\C[G]}(V_{\breve{\chi}}, \C \otimes_{\Z} Y_{K, S})) , 
\end{align} 
where $Y_{K,S}$ denotes the free abelian group on $S$. 

\begin{rem}
Equation (\ref{eqn_order_ArtinLseries}) implies that $r_{S, \chi}=r_{S, \chi^{\alpha}}$ for all automorphisms $\alpha$ of $\C$. If $\psi \in \Ir_{p}(G)$, then we can therefore set $r_{S, \psi}:=r_{S, \chi}$, where $\chi$ is \emph{any} character in $\Ir(G)$ with $\chi^{j}=\psi$ for some field isomorphism $j\colon \C \cong \C_{p}$. In what follows, we will usually identify a $p$-adic character $\psi$ with some fixed $\chi \in \textup{Ir}(G)$. 
\end{rem}

For each character $\chi$ in $\Ir^{+}_{p}(G)$ we write $L_{p, S}(s,\chi)$ for the $S$-truncated Deligne-Ribet $L$-function (cf. \cite{DeRi80}) as discussed by Greenberg in \cite{Gre83}. 
We write $\omega_R$  for the Teichmüller character $\omega_{R}\colon G_{R} \longrightarrow \Z_{p}^{\times}$. Recall that for $\chi \in \Ir^{-}_{p}(G)$ the character $\chi \omega_{R}$ is totally even.

\begin{conjecture}{(Gross order of vanishing conjecture)}\\
\label{conj_ordervanishconj_Gross}
For each $\chi \in \Ir^{-}_{p}(G)$ we have
\[
\ord_{s=0} L_{p, S}(s,\breve{\chi} \omega_{R}) = \ord_{s=0} L_{S}(s, \chi), 
\]
where we again identify $\chi$ with a character in $\Ir^-(G)$ by using some fixed field isomorphism $j\colon \C \cong \C_{p}$.
\end{conjecture}

\begin{rem} \label{rem:gvc} $\;$ \\ \vspace{-4mm} 
  \begin{compactenum}[(1)] 
  \item In the classical formulation of the conjecture, $S = S_\infty(R) \cup S_p(R)$.  We will write $\GVC(K/R, \chi)$ for ``the Gross order of vanishing conjecture for the extension $K/R$ and the character $\chi \in \Ir_{p}^{-}(\Gal(K/R))$'' in this classical setting. 
  \item Let $R$ be a totally real field, and let ${\chi: G_{R} \longrightarrow \overline{\Q}^{\times}}$ be a totally odd finite (one-dimensional!) character. We let $K=\overline{R}^{\ker(\chi)}$ (which is a CM-field) and ${G:= \Gal(K/R)}$.  Then \cite{ChDa14} and \cite{Spi14} independently show that for each ${\chi \in \Ir^{-}_{p}(G)}$ we have
\[
\ord_{s=0} L_{p, S}(s,\breve{\chi} \omega_{R}) \geq \ord_{s=0} L_{S}(s, \chi). 
\] 
  \item \cite[Theorem~3.1~i]{Bur18}, which is proved via a different approach, implies that more generally for arbitrary finite CM-extensions $K$ of $R$ and each $\chi \in \Ir^{-}_{p}(G)$ we have
\[
\ord_{s=0} L_{p, S}(s,\breve{\chi} \omega_{R}) \geq \ord_{s=0} L_{S}(s, \chi).
\]
 \item   By using Brauer Induction one obtains: Conjecture~\ref{conj_ordervanishconj_Gross} is valid if and only if it is valid for all $L$-functions of the form $L_{p, S_{M}}(s,\breve{\phi} \omega_{M})$, where $M$ is any  totally real intermediate field of $K/R$ and $\phi$ is a one-dimensional character in $\Ir_{p}^{-}(\Gal(K/M))$. This is \cite[Remark~2.2]{Bur18}.
\end{compactenum}
\end{rem}

Now we give a list of all (unconditional) results on  Conjecture~\ref{conj_ordervanishconj_Gross} (to the knowledge of the authors). 
\begin{rem} $\;$ \\ \vspace{-4mm} \label{rem:known_GV}\begin{compactenum}[(1)] 
  \item If $r_{S, \chi}=0$, then $\GVC(K/R, \chi)$ holds for $\chi \in \Ir^{-}(G)$ by \cite[p.~989]{Gro81}. 
    \item In \cite{Ems83} $\GVC(K / \Q, \chi)$ is shown for $K / \Q$ abelian and all $\chi \in \Ir^-(G)$.
  \item If $r_{S, \chi}=1$, then $\GVC(K/R, \chi)$ holds for $\chi \in \Ir^{-}(G)$ by \cite[Prop.~2.13]{Gro81} and a result of \cite{Bur18}, as explained in Remark~\ref{rem_final} (2).
\end{compactenum}
\end{rem}

\section{The equivalence of the conjectures} \label{section:equivalence} 

In this chapter we prove Theorem~A from the Introduction, which establishes the equivalence between our characterwise versions of the Gross-Kuz'min conjecture and the order of vanishing conjecture of Gross under a standard hypothesis.

Using the notation from Section \ref{section:Gross_vanishing}, recall that $K$ is a finite Galois CM-extension of a totally real number field $R$, $p$ is an odd prime and $G = \Gal(K/R)$. For any fixed character ${\chi \in \textup{Ir}_p^-(G)}$, we let (by abuse of notation) $\chi \in \textup{Ir}^-(G)$ be a character which corresponds to $\chi$ under some fixed isomorphism $j \colon \C \cong \C_p$, and we define ${R_\chi := \overline{R}^{\ker(\chi)}}$, where $\ker(\chi) = \{ g \in G \mid \chi(g) = \chi(1)\}$. Then $R_\chi \subseteq K$, and $R_\chi$ is a CM-field, since $\chi$ is a totally odd character. We let $H := R_\chi(\mu_p)$ and $H_\infty := H \cdot R_\infty$, where $\mu_p$ denotes the group of $p$-th roots of unity and $R_\infty$ denotes the cyclotomic $\Z_p$-extension of $R$. Note that $H$ is also a CM-field. Finally, we define $\Gamma := \Gal(H_\infty/H)$ and we identify $\Lambda := \Z_p\llbracket \Gamma \rrbracket$ with the ring of power series $\Z_p\llbracket T \rrbracket$. 

\begin{namedthm*}{Theorem~A} \label{thm:equivalence}
Let $K$ be a finite Galois CM-extension of a totally real field $R$ with Galois group $G$, let $p$ be an odd prime and $\chi \in \Ir^{-}_{p}(G)$. Assume that 
\begin{align} \label{eq:bedingung} K \cap R_\infty = R. \end{align} 
Then $\GKC(K/R,\chi)$ holds if and only if $\GVC(K/R, \chi)$ holds.
\end{namedthm*} 

In view of Remark~\ref{rem:chi-components}, we derive the following  
\begin{cor} \label{cor:equivalence} 
  Let $K$ be a CM-field, and let $p \ne 2$. Then $\GKC^-(K)$ holds if and only if for \emph{any} totally real subfield $R$ of $K$ satisfying ${K \cap R_\infty = R}$, $\GVC(K/R, \chi)$ holds for each character $\chi \in \Ir_p^-(\Gal(K/R))$. 
\end{cor} 

\begin{proof}[Proof of Theorem~A] 
	The proof splits into two parts, the first of which is purely algebraic. The second part involves as main ingredient the Iwasawa main conjecture. 
	
	Fix some $\chi$, and let $R_\chi$, $H$ and $H_\infty$ be defined as above. Let $L_{\infty}$ be the maximal unramified abelian pro-$p$-extension of $H_{\infty}$. Then ${X = \Gal({L_\infty}/{H_\infty})}$ is a $\Gal({H_\infty}/R)$-module under the usual conjugation action. 
	Note that, via class field theory, we have a $\Gal({H_\infty}/R)$-isomorphism 
	\[ X =  \Gal({L_\infty}/{H_\infty})  \cong \varprojlim A_n, \]
	where the limit is taken using the norm maps $A_{n + 1} \longrightarrow A_n$ (here  $A_n=\mathrm{Cl}(H_{n})_{p}$, where $H_{n}$ is the $n$-th layer of the $\Z_p$-extension $H_\infty/H$). 

Recall from Section \ref{section:gross_kuzmin} the definitions of $B = \varprojlim B_n$, $A_n' = A_n/B_n$ and ${A' = \varprojlim A_n'}$. If $L_\infty'$ denotes the maximal subextension of $L_\infty/H_\infty$ in which all the primes above $p$ are totally split, then class field theory induces isomorphisms ${B \cong \Gal(L_\infty/L_\infty')}$ and ${A' \cong \Gal(L_\infty'/H_\infty)}$. 

Analogously, Iwasawa modules can be defined for the cyclotomic $\Z_p$-extension $M_\infty$ of any number field $M$. In order to distinguish the corresponding Iwasawa modules, we will write $A^{(M_\infty)}$, $(A')^{(M_\infty)}$, etc. 

As a first step, we will show that $\GVC(K/R, \chi)$ holds if and only if $\GVC(H/R, \chi)$ holds (our motivation for proving this reduction is that the Iwasawa main conjecture, which will be used in the second part of the proof, is formulated for the $\Z_p$-extension $H_\infty/H$ rather than for $K_\infty/K$). 

\begin{lemma} \label{lemma:equivalences} 
   Let $\chi \in \Ir^-(G)$, and let $M$ be any finite normal CM-extension of $R$ containing $R_\chi$. We assume that $M \cap R_\infty = R$. 
   Then $\GKC(M/R, \chi)$ holds if and only if $\GKC(R_\chi/R, \chi)$ holds. 
   
   In particular, then the following conjectures are equivalent: \begin{compactenum}[(a)] 
       \item $\GKC(K/R, \chi)$, 
       \item $\GKC(R_\chi/R,\chi)$, 
       \item $\GKC(H/R, \chi)$. 
   \end{compactenum} 
\end{lemma} 
\begin{proof} 
   Let $M_\infty$ and $(R_\chi)_\infty$ denote the cyclotomic $\Z_p$-extensions of $M$ and $R_\chi$, and write $\mathcal{G} = \Gal(M/R)$ for this proof. Using the canonical surjection ${\Gal(\overline{R}/R) \twoheadrightarrow \mathcal{G}}$, we can view $\chi$ as a character on $\mathcal{G}$. $\GKC(M/R, \chi)$ concerns the module 
   $$ \frac{\chi(1)}{|\mathcal{G}|} \cdot \sum_{\sigma \in \mathcal{G}} \chi(\sigma) \cdot \sigma^{-1} \cdot \left((A')^{(M_\infty)} \otimes_{\Z_p} \overline{\Q}_p \right). $$ 
   Since $\tilde{\mathcal{G}} := \Gal(M/R_\chi) = \ker(\chi|_M)$, $\chi$ factors through $\mathcal{G}/\tilde{\mathcal{G}} = \Gal(R_\chi/R)$. If $I := \{\sigma_1, \ldots, \sigma_l\}$ denotes a set of representatives for the cosets of $\mathcal{G}/\tilde{\mathcal{G}}$, then ${\chi(\sigma_i \rho) = \chi(\sigma_i)}$ for each $i \in \{1, \ldots, l\}$ and every $\rho \in \tilde{\mathcal{G}}$. Therefore the above module can be written as 
   \begin{align} \label{eq:gkc-reduction}  \frac{\chi(1)}{|\mathcal{G}|} \cdot \sum_{i = 1}^l \chi(\sigma_i) \cdot \sigma_i^{-1} \cdot \left( \sum_{\rho \in \tilde{\mathcal{G}}} \rho^{-1} \right) \cdot \left((A')^{(M_\infty)} \otimes_{\Z_p} \overline{\Q}_p \right) . \end{align} 
 
   The inner sum in \eqref{eq:gkc-reduction} yields a group ring element corresponding to the norm map $N_{M/R_\chi}$, and 
   $$\Gal(M/R_\chi) \cong \Gal(M_\infty/(R_\chi)_\infty)$$ because $M \cap (R_\chi)_\infty = R_\chi$. Note that  
   $$((A')^{((R_\chi)_\infty)})^{[M:R_\chi]} \subseteq  N_{M_\infty/(R_\chi)_\infty}((A')^{(M_\infty)}) \subseteq (A')^{((R_\chi)_\infty)}, $$ i.e. the characteristic power series of these three $\Lambda$-modules differ by at most a power of $p$. This shows that $\GKC(R_\chi/R, \chi)$ holds if and only if $\GKC(M/R, \chi)$ holds. 
\end{proof} 

The following proposition is due to Sinnott. 
\begin{prop}{\cite[Prop. 6.1]{FeGr81}}
\label{prop_FeGr81_Prop6.1}
\begin{compactenum}[(a)] 
    \item $$B^- \cong \Gal(L_{\infty}/L_{\infty}^{\prime})^- \cong \Big(\bigoplus_{\substack{w \mid p \\ \textup{in } H_\infty}} w \Z_p \Big)^- $$ 
    as $\Z_p\llbracket\Gal(H_\infty/R)\rrbracket$-modules. 
    \item The module $(A')^- \cong \Gal(L_\infty'/H_\infty)^-$ does not contain any nontrivial finite $\Lambda$-submodules.
\end{compactenum}
\end{prop}

Recall that $\Gal(H_\infty/R)$ acts on $X = \Gal(H_\infty/L_\infty)$. We have ${R_{\chi} \cap R_{\infty} = R}$ by assumption \eqref{eq:bedingung}. Therefore also ${R_\infty \cap H = R}$ (because $p$ does not divide the degree $[H:R_\chi]$), and 
$$ \Gal(H_\infty/R) \cong \Gal(H_\infty/H) \times \Gal(H/R) . $$ 
Let $\Delta := \Gal(H/R)$. We write $\Delta_w \subseteq \Delta$ for the decomposition group of any ${w \in S_p(H)}$ and $G_w \subseteq G = \Gal(K/R)$ for the decomposition group of a ${w \in S_p(K)}$. 

Fix a prime $v$ of $R$ dividing $p$, and consider the module 
$$ B_v := \bigoplus_{\substack{w \mid v \\ \text{in }\, H_\infty}} w \overline{\Q}_p. $$ 
Recall that $B^- \otimes_{\Z_p} \overline{\Q}_p \cong \left(\bigoplus_{v \mid p} B_v \right)^-$ by Proposition  \ref{prop_FeGr81_Prop6.1},(a). 

In the following lemma, we fix a prime $w_1$ of $H$ above $v$ and a prime $w_2$ of $K$ dividing $v$ such that $w_1 \cap R_\chi = w_{2} \cap R_\chi$, and we denote by $p^n$, for $n \in \N$, the number of primes of $H_\infty$ dividing $w_1$. Finally, let $\omega_n(T) = (T+1)^{p^n}-1$, $n \in \N$. 

\begin{lemma} \label{lemma:3.4} 
   Let $\chi \in \Ir^-(G)$ be a character with representation $V_\chi$. Then $$\varepsilon_\chi \cdot B_v \cong (\overline{\Q}_p \otimes_{\Z_p} \Lambda/(\omega_n))^{\dim_\C(V_\chi^{G_{w_2}}) \cdot \chi(1)}. $$ 
\end{lemma} 

\begin{rem} \label{rem:zu_Lemma_3.4} $\,$ \\ \vspace{-2mm} \begin{compactenum}[(1)] 
 \item Since $R_\chi$ is contained in both $H$ and $K$, we can consider $\chi$ as a character on $\Delta$ and on $G$.
  \item The number $p^n$ of primes of $H_\infty$ lying above $w_1$ does not depend on the choice of $w_1$ because the extension $H_\infty/R$ is normal. 
  \item Similarly, the dimension of the vector space $V_\chi^{G_{w_2}}$ does not depend on the choice of the prime $w_2$ of $K$ dividing $v$. 
  \item Actually $V_\chi^{G_{w_2}} = V_\chi^{\Delta_{w_1}}$ if $w_1 \cap R_\chi = w_2 \cap R_\chi$: if $D_1$ denotes the decomposition field of $w_1$ in $H/R$ and $D_2$ denotes the decomposition field of $w_2$ in $K/R$, then $D_1 \cap R_\chi = D_2 \cap R_\chi$, because the prime $w_i \cap D_i \cap R_\chi$ does not split at all in $R_\chi$, for $i \in \{1,2\}$. Now $\Gal(K/R_\chi)$ and $\Gal(H/R_\chi)$ are contained in the kernel of $\chi$. Therefore 
  $$ V_\chi^{\Delta_{w_1}} = V_\chi^{\Gal(H/D_1) \cdot \Gal(H/R_\chi)} = V_\chi^{\Gal(H/(R_\chi \cap D_1)}, $$ 
  and similarly for $V_\chi^{G_{w_2}}$. 
\end{compactenum} 
\end{rem} 

\begin{proof}[Proof of Lemma \ref{lemma:3.4}]  
  We first consider the extension $H/R$ of number fields. Let $$B_{v, H} := \bigoplus_{\substack{w \mid v \\ \textup{in } \, H}} w \overline{\Q}_p. $$ 
  The $\overline{\Q}_p[\Delta]$-module 
  $\varepsilon_\chi \cdot B_{v, H}$  
  is the direct sum of the $\overline{\Q}_p[\Delta]$-submodules of $B_{v, H}$ which are isomorphic to the representation space $V_\chi$. Now $V_\chi$ has dimension $\chi(1)$, and the multiplicity of such a representation in $B_{v, H}$ is given by $\langle \chi, \psi \rangle_{\Delta}$, where $\psi$ denotes the character of the $\overline{\Q}_p[\Delta]$-module $B_{v,H}$. 
  
  By definition of $B_{v, H}$, $\psi = \textup{Ind}_{\Delta{w_1}}^\Delta \mathds{1}_{\Delta_{w_1}}$, and therefore Frobenius reciprocity implies that 
  $$ \langle \chi , \psi \rangle_\Delta \, = \, \langle \chi|_{\Delta_{w_1}}, \mathds{1}_{\Delta_{w_1}} \rangle_{\Delta_{w_1}} \, = \, \dim_{\C}(V_\chi^{\Delta_{w_1}}) $$ 
  (see also \cite[p.~24]{Tat84}). This means that we have an isomorphism of $\overline{\Q}_p$-vector spaces 
  $$ \varepsilon_\chi \cdot \Big( \bigoplus_{\substack{w \mid v \\ \text{in }\, H}} w \overline{\Q}_p \Big) \, \cong \, \overline{\Q}_p^{\dim_\C(V_\chi^{\Delta_{w_1}}) \cdot \chi(1)}. $$  
 
 Now we go up the Iwasawa towers. 

 For every prime $w_0 \mid v$ of $H$, we note that 
 $$\bigoplus_{\substack{w \mid w_0 \\ \textup{in} \, H_\infty}} w \overline{\Q}_p $$ 
 is a $\Lambda$-module, and a $\overline{\Q}_p$-vector space of dimension $p^n = \textup{deg}(\omega_n)$, by the definition of $n$. 
 Since $\omega_n$ annihilates each prime of $H_\infty$ lying above $w_0$, it follows that we have an isomorphism 
 $$ \bigoplus_{\substack{w \mid w_0 \\ \textup{in} \, H_\infty}} w \overline{\Q}_p \, \cong \, \overline{\Q}_p \otimes_{\Z_p} \Lambda/(\omega_n)$$ 
 of $\overline{\Q}_p \otimes_{\Z_p} \Lambda$-modules. In particular, this implies that
$$ \varepsilon_\chi \cdot B_v \, \cong \, (\overline{\Q}_p \otimes_{\Z_p} \Lambda/(\omega_n))^{\chi(1) \cdot \dim_\C(V_\chi^{\Delta_{w_1}})}. $$

 As we already pointed out in Remark \ref{rem:zu_Lemma_3.4}(4), we may replace the decomposition group $\Delta_{w_1}$ by the decomposition group $G_{w_2} \subseteq G$ in the above formula if ${w_1 \cap R_\chi = w_2 \cap R_\chi}$; this proves the lemma. 
\end{proof}

Now we start with the analytic part of the proof. Recall that ${r_{S,\chi} = \textup{ord}_{s=0}L_S(s, \chi)}$, as in Section~\ref{section:Gross_vanishing}. 
\begin{lemma} \label{lemma:tate} 
  For each $\chi \in \textup{Ir}^-(G)$, we have 
  $$ r_{S,\chi} = \sum_{v \in S} \dim_{\C}(V_\chi^{G_w}), $$ 
  where $V_\chi$ denotes the representation space of $\chi$, and where $w \in S(K)$ denotes any prime dividing $v$, respectively. 
\end{lemma} 
\begin{proof} 
  See \cite[Chapter I, Proposition 3.4]{Tat84}. 
\end{proof} 

Note: the infinite places do not contribute to the above sum, since $K$ is a CM-field. 
\vspace{2mm}

Recall that $X = \Gal(L_\infty/H_\infty)$. 

Let $V_X := X \otimes_{\Z_p} \overline{\Q}_p$, and let $V_X^{(\chi)} = \varepsilon_{\chi} \cdot V_X$ be the eigenspace corresponding to the action of $\Gal(H/R)$ via $\chi$. We define $f_{X,\chi} (T)$ to be the characteristic polynomial of $\gamma - 1$, where $\gamma$ is a fixed generator of $\Gamma$ (cf. also Section \ref{section:gross_kuzmin}). 

Since the multiplicity of each zero of $f_{X,\chi}(T)$ (in the algebraic closure $\overline{\Q}_p$) is a multiple of $\chi(1)$, we can define a polynomial $g_{X,\chi}(T) \in \Z_p[\chi][[T]]$ as the monic polynomial that satisfies $f_{X,\chi}(T)=g_{X,\chi}(T)^{\chi(1)}$. 

Next we look at the connection to $p$-adic $L$-functions. Let ${c \colon \Gamma \longrightarrow \Z_{p}^{\times}}$ be the cyclotomic character and set $u:=c(\gamma)$.  Recall that $S=S_{\infty}(R) \cup S_{p}(R)$. Then for each non-trivial totally even  one-dimensional character $\psi$ there exists (e.g. by \cite{DeRi80}) a unique power series $G_{\psi}(T) \in \Z_{p}[\psi]\llbracket T  \rrbracket$  such that 
\begin{align}
\label{eqn_Lp_powerseries}
L_{p,S}(1-s, \psi)= \frac{G_{\psi}(u^{s}-1)}{H_{\psi}(u^{s}-1)} 
\end{align}
for all $s \in \Z_{p}$, where $H_{\psi}(T) = (\psi(\gamma)(1+T)-1)$ if $R_{\psi} \subseteq R_{\infty}$ and $H_{\psi}(T) = 1$ otherwise.

Assume now that $\psi$ is a character of arbitrary dimension with $R_{\psi}$ totally real. With the help of Brauer Induction one can use the one-dimensional case to define a $p$-adic $L$-function  satisfying the desired interpolation property.
Moreover, one finds power series  $G_{\psi}(T)$ and $H_{\psi}(T)$ satisfying an equation analogous to (\ref{eqn_Lp_powerseries}).

In this general situation, $G_{\psi}(T)$ a priori is contained only in the quotient field of $\Z_{p}[\psi]\llbracket T \rrbracket$, but using \cite[Prop.~5]{Gre83}, Wiles showed (cf. \cite[Thm.~1.1]{Wil90}), by proving the one-dimensional main conjecture, that if $p$ is odd and $\psi$ is such that $R_{\psi}$ is totally real, then 
   \[
   G_{\psi}(T) \in \Z_{p}[\psi]\llbracket T \rrbracket \otimes \Q_{p}.
   \]

Using  the Iwasawa main conjecture over totally real fields for one-dimensional characters (i.e. \cite[Thm.~1.2]{Wil90}), Brauer Induction and the well-known functoriality properties of the power series involved, one can deduce the

\begin{thm}{(Iwasawa main conjecture)}\\
\label{thm_IMCarbdeg}
Let $p$ be an odd prime and let $\chi$ be a totally odd character of arbitrary dimension satisfying $R_{\chi} \cap R_{\infty}=R$. Then we have
   \[
   g_{X,\chi}(T)=\{G_{\check{\chi} \omega_{R}} ( u(1+T)^{-1} - 1) \}^{\ast}. 
   \]
\end{thm} 
Here $\ast$ means the following: if $\lambda \in \Lambda$, then by the Weierstra{\ss} Preparation Theorem $\lambda$ is associated to a power of $p$ times a distinguished polynomial, and we denote by $\lambda^\ast$ this polynomial.

From Theorem~\ref{thm_IMCarbdeg} we immediately obtain
\begin{align} \label{eq:main-conjecture} 
   \chi(1) \cdot \ord_{s=0}L_{p, S}(s, \breve{\chi} \omega_{R})  = \ord_{T=0} f_{X,\chi}(T) . 
\end{align}

Now we can put the components of the proof together. By the definitions and in view of Proposition \ref{prop_FeGr81_Prop6.1}, we have an exact sequence 
\begin{align} \label{eq:exakte_sequenz}  0 \longrightarrow \bigoplus_{v \in S} \varepsilon_\chi B_v \longrightarrow \varepsilon_\chi (A \otimes_{\Z_p} \overline{\Q}_p)  \longrightarrow \varepsilon_\chi (A' \otimes_{\Z_p} \overline{\Q}_p) \longrightarrow 0.  \end{align}  
Lemma \ref{lemma:3.4} implies that 
$$\varepsilon_\chi B_v \, \cong \, (\overline{\Q}_p \otimes_{\Z_p} \Lambda/(\omega_{n(v)}(T)))^{\chi(1) \cdot \dim_\C(V_\chi^{G_{w_2(v)}})}. $$ 
Here $w_1(v)$ and $w_2(v)$ are fixed primes of $H$, respectively, $K$ dividing $v$, and $n(v)$ denotes the number of primes of $H_\infty$ above $w_1(v)$, as in the proof of Lemma~\ref{lemma:3.4}.

By the multiplicativity of characteristic power series, we obtain that the exact power of $T$ dividing 
$f_{X,\chi} = f_{A, \chi}$ equals 
\begin{align} \label{eq:T-ord} \textup{ord}_{T = 0} f_{A', \chi}(T) + \sum_{v \in S} \chi(1) \cdot \dim_\C(V_\chi^{G_{w_2(v)}}) \; \stackrel{(\ref{lemma:tate})}{=} \ord_{T=0}f_{A', \chi}(T) + \chi(1) \cdot r_{S,\chi} , \end{align}  
since $\ord_{T = 0}(\omega_m(T)) = 1$ for each $m \in \N$.

The first summand in \eqref{eq:T-ord} is zero if and only if $\GKC(H/R, \chi)$ holds. In view of Lemma \ref{lemma:equivalences}, this is equivalent to the validity of $\GKC(K/R, \chi)$. On the other hand, by the main conjecture, $\GVC(K/R, \chi)$ is true if and only if $$\textup{ord}_{T=0}f_{X,\chi}(T) = \chi(1) \cdot r_{S,\chi}.$$ 
This concludes the proof. 
\end{proof} 

\clearpage
\begin{rem} \label{rem_final}  $\;$ \\ \vspace{-4mm} \begin{compactenum}[(1)] 
    \item Equations \eqref{eq:main-conjecture} and \eqref{eq:T-ord} reprove the result 
$$ \ord_{s=0}L_{p,S}(s, \breve{\chi} \omega_R) \ge \ord_{s=0}L_S(s, \chi)$$ 
from \cite{Bur18} (cf. Remark \ref{rem:gvc},(3)), under our assumption ${K \cap R_\infty = R}$. 
\item As already mentioned in the Introduction, one can obtain Corollary~\ref{cor:equivalence} more generally by using \cite{Bur18} and \cite{Kol91}. Indeed, in \cite{Gro81} it is proved that the non-vanishing of the 'Gross regulator' is equivalent to Conjecture~1.15 in \cite{Gro81}. A 'minus-version' of \cite[Theorem~1.14]{Kol91} then shows that $\GKC^{-}(K)$ is equivalent to  Conjecture~1.15 in \cite{Gro81}. Now Burns proved in \cite{Bur18} that the 'Gross regulator' is non-zero if and only if $\GVC(K/R, \chi)$ holds. Combining these results, one obtains our Corollary~\ref{cor:equivalence} without the assumption ${K \cap R_{\infty}=R}$. 
    \item Conjecture~\ref{conj:GKC-minus} can also be formulated for $p=2$. Most of the proof of Theorem~A can be adapted in order to cover also this case. 
    
    The only step of the proof which does not carry over is the present lack of a main conjecture in the $p=2$ setting. 
\end{compactenum} 
\end{rem} 

Using Theorem~A and Corollary \ref{cor:equivalence}, one can try to derive from the known results about $\GKC$ new cases of $\GVC$; unfortunately, it turns out that the classically known \lq trivial' cases from Remark~\ref{rem_GKC_known},(1) yield only situations where $r_{S,\chi} \le 1$, and for which therefore also the $\GVC$ is known. %

\begin{cor} 
  Let $K$ be a normal CM-extension of a totally real number field $R$ and let $p \ne 2$. Suppose that $R$ contains exactly one prime above $p$ and that $K \cap R_\infty = R$. Then $\GKC^-(K)$ holds. 
\end{cor} 
\begin{proof} 
  Indeed, since $r_{S, \chi} \le 1$ for each $\chi \in \Ir^-(G)$, $G = \Gal(K/R)$, it follows from \cite{DaKaVe18} that $\GVC(K/R, \chi)$ holds for every $\chi \in \Ir^-(G)$. This proves $\GKC^-(K)$ in view of Corollary~\ref{cor:equivalence}.
\end{proof} 
Here again we note that this result can be obtained more generally, i.e. without the condition $K \cap R_{\infty} = R$, via the work of Burns in \cite{Bur18}, as explained in Remark~\ref{rem_final}(2).

Having in mind the known results mentioned in Section \ref{section:Gross_vanishing}, it is of particular interest to produce examples where $r_{S, \chi}$ is large. Here Corollary \ref{cor:equivalence} can be very helpful in order to derive from a known instance of $\GVC(K/R, \chi)$ for some $\chi$ with $r_{S, \chi} > 0$ examples where $\GVC$ is known for some character with larger vanishing order: 
\begin{rem} \label{rem:hochgehen} 
  Let $K$ denote a Galois CM-extension of a totally real number field $R$. Suppose that $\GKC^-(K)$ is known and that we can enlarge $R$ and choose some totally real number field $\tilde{R} \subseteq K$ which strictly contains $R$. In view of Tate's formula~\ref{lemma:tate}, the value of $r_{S, \chi}$, for suitable $\chi$, can grow if we consider characters of the group $\Gal(K/\tilde{R})$. If, for example, $K/R$ is abelian, then Tate's formula just says that $r_{S, \chi}$ equals the number of primes of $R$ above $p$ which split completely in $R_\chi$. Now suppose that the primes of $R$ dividing $p$ are totally split in $\tilde{R}$. Then each character $\tilde{\chi} \in \Ir^-(\Gal(K/\tilde{R}))$ satisfies 
  $$ r_{S, \tilde{\chi}} = |S_p(\tilde{R})| = [\tilde{R}:R] \cdot |S_p(R)|. $$ 
\end{rem}

In the next section, we will prove $\GKC$ in new cases, which then also yield new instances of the Gross order of vanishing conjecture.

 \section{Proving the conjectures in new cases} \label{section:examples} 
 In this chapter we first give an equivalent formulation of the minus Gross-Kuz'min conjecture.

Then we describe two approaches for the construction of families of abelian extensions $K/R$ such that $\GKC(K/R, \chi)$ and $\GVC(K/R, \chi)$ hold for some $\chi \in \Ir^-(\Gal(K/R))$ with large vanishing order $r_{S, \chi}$ of the corresponding Artin $L$-function. The first approach makes use of a relation to Leopoldt's conjecture, and the second one is based on an application of the Brumer/Baker Theorem. 
 
 Let $p \ne 2$ be a prime, let $K$ be a CM-field with maximal totally real subfield $K^+$, let $K_\infty^+$ be the cyclotomic $\Z_p$-extension of $K^+$. We write $K_\infty^+ = \bigcup K_n^+$, and we let ${K_\infty := K_\infty^+ \cdot K = \bigcup K_n}$, where $K_n = K_n^+ \cdot K$ for each $n \in \N$. We assume that each prime of $K^+$ dividing $p$ is totally ramified in $K_\infty^+$. Let $r$ denote the number of primes above $p$ in $K^+$ which split in $K/K^+$. We identify $\Gamma_n := \Gal(K_n/K)$ with $\Gal(K_n^+/K^+)$ for every $n \in \N$. Finally, we recall that $E_K'$ denotes the group of $p$-units of $K$. 
 \begin{lemma} \label{lemma:chevalley-}
 	Under the above assumptions, $\GKC^{-}(K)$ holds if and only if there exists a constant $C \in \N$ such that 
 	\begin{align} \label{eq-cond} \frac{[E_{K^+}' : (N_n(K_n^+) \cap E_{K^+}')]}{[E_K' : (N_n(K_n) \cap E_K')]} \le \frac{1}{p^{rn-C}},  \end{align} 
 	i.e. 
 	$$[E_K' : (N_n(K_n) \cap E_K')] \ge p^{rn - C} \cdot [E_{K^+}':(N_n(K_n^+) \cap E_{K^+}')], $$ 
 	for all $n \in \N$. 
 \end{lemma} 
 \begin{proof} 
 	Fix a level $n$. We use Chevalley's Theorem (see \cite[Lemma 4.1 in Chapter 13]{lang}), or more precisely the variant for $A'$ and $p$-units due to Gras (see \cite{gras}), once for the cyclic extension $K_n/K$ and once for the extension $K_n^+/K^+$. 
 	This yields: 
 	\begin{align} \label{chevalley} \frac{|(A_n')^{\Gamma_n}|}{|((A_n')^+)^{\Gamma_n}|} = \frac{|A_0'|}{|(A_0')^+|} \cdot \frac{e(K_n/K)}{e(K_n^+/K^+)} \cdot \frac{[E_{K^+}' : (N_n(K_n^+) \cap E_{K^+}')]}{[E_K' : (N_n(K_n) \cap E_K')]}. \end{align} 
 	Here $e(K_n/K)$ denotes the product of the ramification indices in $K_n/K$ of all the primes of $K$, and $e(K_n^+/K^+)$ is defined analogously. 
	
 	Now the quotient on the left hand side yields the order of $((A_n')^-)^{\Gamma_n}$. 
 	Recall that $\GKC^-(K)$ holds if and only if the orders of these groups remain bounded as $n \to \infty$.  
	
 	Since all the primes dividing $p$ are totally ramified in $K_\infty/K$ and in $K^+_\infty/K^+$, the quotient of the ramification numbers equals $p^{rn}$, where we recall that $r$ denotes the number of primes of $K^+$ dividing $p$ which split in $K/K^+$. The statement of the lemma follows: in order to have stabilisation of the orders on the left hand side, it is (see \cite[Theorem~2.1]{T-ranks}) necessary and sufficient that there exists an index $n_0 \in \N$ such that $|((A_m')^-)^{\Gamma_m}| = |((A_{n_0}')^-)^{\Gamma_{n_0}}|$ for all $m \ge n_0$. 
 \end{proof} 

\begin{cor} \label{cor:GK-trivial} 
  In particular, $\GKC^{-}(K)$ holds if no prime of $K^+$ dividing $p$ splits in $K$. 
 \end{cor} 

 In \cite{green-l-adic}, Greenberg implicitly proves the following reduction theorem;  this result has been reproved and used by Jaulent in \cite{Jau17}. 
 \begin{thm} \label{thm:red_1} 
 	$\GKC^-(K)$ holds if \begin{compactenum}[(a)] 
		\item $K$ contains a CM-subfield $k$ such that the prime $p$ is undecomposed in $K/k$, and 
		\item $\GKC^-(k)$ holds.
	\end{compactenum} 
 \end{thm} 
 \vspace{2mm} 
 
In what follows, we will derive two approaches for constructing finite families of extensions $K/R$ such that $\GKC(K/R, \chi)$ holds for some $\chi \in \Ir^-(\Gal(K/R))$ with large vanishing order $r_{S, \chi}$. The first approach is related to Leopoldt's conjecture. 
\begin{prop} 
\label{prop_Leopoldt_totallysplit}
  Let $K$ be a CM-field, and suppose that Leopoldt's conjecture holds for $K$. If the prime $p$ is totally split in $K/\Q$, then $\GKC^-(K)$ holds. 
\end{prop} 
\begin{proof} 
   Let $A = A^{(K_\infty)}$ be defined as in Section \ref{section:gross_kuzmin}. First, it follows that Leopoldt's conjecture also holds for $K^+$, i.e. the quotient $A^+/(T \cdot A^+)$ is finite. Moreover, there exist exactly $2 r_2(K)$ primes in $K$ and $r_2(K)$ primes in $K^+$. Since Leopoldt's conjecture holds for $K$, we know that 
   $$ \textup{ord}_{T = 0}(f_{A}^-(T)) = \rg_{\Z_p}(A/(T \cdot A)) = r_2(K). $$ 
   Now we consider the exact sequence 
   $$ \xymatrix{0 \ar[r] & \bigoplus\limits_{i=1}^r \Lambda/(\omega_{n_i}(T)) \ar[r] & A^- \ar[r] & A'^- \ar[r] & 0} $$ 
   which follows from the exact sequence  \eqref{eq:exakte_sequenz} in the proof of Theorem~A (equivalence of the two conjectures) by specialising to the case $R = K^+$. Here $r = r_{S, \chi}$ for the only non-trivial character of $\Gal(K/K^+)$, and the integers $n_i$ are defined as in Lemma~\ref{lemma:3.4}. Since $p$ is totally split in $K/\Q$, it follows that $r$ equals the number of primes of $K^+$ above $p$, i.e. $r = r_2(K)$. 
   This proves that $\ord_{T=0} f_{A'}^-(T) = 0$. 
\end{proof} 

Now we can use the existing literature on Leopoldt's conjecture for deriving cases of the $\GKC$.

\begin{thm}{\cite[Thm.~2.1~c)]{Kli90}}
\label{thm_Thm2.1c_Kli90}
Let $K/ \Q$ be an imaginary Galois extension with group $G$, and let $\tau \in G$ denote the complex conjugation, well defined up to conjugation in $G$. Then Leopoldt's conjecture is true for all $p$, provided that for all irreducible characters $\chi$ of $G$ the condition 
\begin{align}
\label{cond_Thm2.1c_Kli90}
\chi(1)+\chi(\tau) \leq 2
\end{align}
holds.
\end{thm}
In the case which is of interest to us, namely $K$ being a CM-field, this condition is quite restrictive. In fact, Klingen shows in \cite[Lemma~2.3]{Kli90} that under the condition that $K$ is a CM-field hypothesis (\ref{cond_Thm2.1c_Kli90}) is satisfied if and only if $G/ \left\langle \tau \right\rangle$ is abelian. So these cases can also be covered by combining the classical facts that Leopoldt's conjecture is proven for abelian extensions over $\Q$ and that Leopoldt's conjecture holds for a CM-field $K$ if and only if it holds for its maximal totally real subfield $K^{+}$ (see e.g. \cite[Cor.~10.3.11~(ii)]{NeScWi08}). The main examples for such extensions which are not abelian over $\Q$ are $D_{4}$- and $Q_{8}$-extensions.

Our interest for these results stems from the following corollary to Theorem~\ref{thm_Thm2.1c_Kli90}. 
\begin{cor}{\cite[Cor.~4.2.2]{Kli90}}
\label{cor_Cor4.2.2_Kli90}
  Let $K/\Q$ be an imaginary Galois extension with Galois group $G$, and let $\tau$ denote the complex conjugation in $G$. Suppose that condition  (\ref{cond_Thm2.1c_Kli90}) holds. Then Leopoldt's conjecture is true for all composita $KL$ of $K$ with a real abelian extension $L$ of $\Q$.
\end{cor}
Now we can show our Theorem~B from the Introduction.
\begin{thm} \label{thm:q8} 
There exist infinitely many $Q_{8}$-and $D_{4}$-CM-extensions $K$ of totally real number fields for which  $\GKC(K)$ holds for infinitely many primes. In particular, for infinitely many primes and infinitely many $D_{4}$- and $Q_{8}$-extensions $K/R$, ${\GVC(K/R, \chi)}$ holds for all characters $\chi \in \Ir^-(\Gal(K/R))$. This includes examples with arbitrarily large vanishing order.
\end{thm}
\begin{proof}
Let $r \in \N$ and  $M/ \Q$ be a $D_{4}$- or $Q_{8}$-CM-extension. We construct $R$ by composing sufficiently many real quadratic number fields, not contained in $M/\Q$, such that $[R: \Q] \geq r$. Then $MR/ \Q$ is a Galois CM-extension which is non-abelian over $\Q$. With Theorem~\ref{thm_Thm2.1c_Kli90} and Corollary~\ref{cor_Cor4.2.2_Kli90} we get that Leopoldt's conjecture holds for $MR$ for all $p$. By the Chebotarev density Theorem there are infinitely many primes  which split completely in the Galois extension $MR/\Q$. For all these primes we can use Proposition~\ref{prop_Leopoldt_totallysplit} to derive  $\GKC^-(MR)$. 
Now the vanishing order of the unique totally odd irreducible character $\chi$ of $\Gal(MR/R)$ is 
$$r_{S, \chi} = 2 \cdot |S_p(R)|$$ 
by Tate's Lemma \ref{lemma:tate}, since $\chi(1)=2$ and $G_w = \{1\}$ for each $w$. For all odd primes we can use Theorem~A in order to deduce $\GVC(MR/R, \chi)$ from $\GKC(MR/R, \chi)$: the condition $MR \cap R_\infty = R$ is satisfied since $p$ does not divide the degree $[MR :R]$, which is a power of $2$. 
\end{proof}

Now we describe the second approach for constructing new examples where $\GKC(K/R, \chi)$ holds for some character $\chi$ with large vanishing order $r_{S, \chi}$.

 \begin{thm} \label{thm:weihnacht}
  Let $K$ be a finite normal CM-extension of a totally real number field $R$, and suppose that \begin{compactenum}[(a)] 
  	   \item there exists some prime $\p$ of $R$ above $p$ which is totally split in $K$, and such that $R_\p = \Q_p$, 
  	   \item $G := \Gal(K/R)$ is abelian. 
  	\end{compactenum} 	
    Let $A' = (A')^{(K_\infty)}$ be defined as in Section \ref{section:gross_kuzmin}. Then 
  	$$ \rg_{\Z_p}((A')^-_\Gamma) \le r - s, $$ 
  	where $r$ denotes the number of primes of $K^+$ which split in $K$, and $[K:R] = 2s$. 
\end{thm} 
This is Theorem~C from the Introduction and we note that hypothesis (a) may also be stated in the following way: there exists a prime $\mathfrak{P} \in S_p(K)$ such that $K_{\mathfrak{P}} = \Q_p$. 

\begin{proof} 
	 We will use arguments from the proof of \cite[Proposition~3]{green-l-adic}. In all what follows, we denote by $H_0$ the group of $p$-units $\alpha$ of $K$ which are divisible only by powers of ideals of $K$ which lie above $\p$, and which satisfy $\tau(\alpha) = \alpha^{-1}$, where $\tau$ denotes the complex conjugation. Let $\p_1, \ldots, \p_s, \overline{\p_1}, \ldots, \overline{\p_s}$ be the primes of $K$ above $\p$. 
	 
	 Let $K_{\p_i}$ be the completion of $K$ at $\p_i$, and let $\varphi_i\colon K \longrightarrow K_{\p_i}$ be the canonical embedding. Note that $K_{\p_i} = \Q_p$ for every $i$, by our assumptions. We define a map 
	 $$\psi \colon H_0 \longrightarrow \mathcal{U} := \prod_{i=1}^s E_{K_{\p_i}} $$ 
	 by 
	 $$ \psi(\alpha) := (\psi_1(\alpha), \ldots, \psi_s(\alpha)), $$ 
	 where $\psi_i(\alpha) := \varphi_i(\alpha) \cdot p^{- v_{\p_i}(\alpha)}$ for every ${i \in \{1, \ldots, s\}}$. 
	 
	 Our proof is based on four claims. 
	 \begin{claim} \label{claim_1} 
	 	$\psi$ is injective. 
	 \end{claim} 
	 \begin{proof} 
	 	If $\psi(\alpha) = 1$, then $\varphi_i(\alpha) = p^{v_{\p_i}(\alpha)}$ for each $i$. We consider $i = 1$. Since both elements $\varphi_1(\alpha)$ and $p^{v_{\p_1}(\alpha)}$ are contained in $K$, and as $\varphi_1$ is injective, it follows that $\alpha = p^{v_{\p_1}(\alpha)}$ in $K$. But $\tau(\alpha) = \alpha^{-1}$, and therefore $v_{\p_1}(\alpha) = 0$ and $\alpha = 1$. 
	 \end{proof} 
	 
	 By Hasse's Norm Theorem, an element $\alpha \in K$ is a norm from the $n$-th layer $K_n$ of the cyclotomic $\Z_p$-extension of $K$ if and only if it is a local norm at each prime dividing $p$. The element $\alpha$ is a local norm in $K_{\p_i}$ if and only if $\psi_i(\alpha) \in E_{K_{\p_i}}^{p^n}$. In particular, $\psi(N_n(K_n)) \subseteq \mathcal{U}^{p^n}$. Therefore 
	 \begin{align} 
	    [(H_0 \cap N_n(K_n)): H_0^{p^n}] \le [(\psi(H_0) \cap \mathcal{U}^{p^n}): \psi(H_0)^{p^n}] .  \label{noname} 
	 \end{align} 
	
	We let $H := \psi(H_0)$, and we denote by $\overline{H}$ the $p$-adic closure of $H$ in $\mathcal{U}$. Recall that $H_0$ is a $\Z$-module of rank $s = \frac{|G|}{2}$. 
	\begin{claim} \label{claim_2} 
		$\rg_{\Z_p}(\overline{H}) = s$. 
	\end{claim} 
	\begin{proof} 
		Consider the matrix $A = (\psi_i(\alpha))_{\alpha \in H_0, i \in \{1, \ldots, s\}}$. If $\rg_{\Z_p}(\overline{H}) < s$, then the dimension of the row space of $A$ is at most $s-1$. This means that the dimension of the kernel of the corresponding $\Z_p$-linear map $f\colon \Z_p^s \longrightarrow \mathcal{U}$ is at least one. In other words, there exist $a_1, \ldots, a_s \in \Z_p$, not all zero, such that 
		\begin{align*} \prod_{i = 1}^s \psi_i(\alpha)^{a_i} = 1 \quad \text{ for each } \; \alpha \in H_0. \end{align*} 
		Taking the logarithm yields 
		\begin{align} \label{extra-gleichung} \sum_{i = 1}^s a_i \log_p(\varphi_i(\alpha)) = 0 \quad \text{ for each $\alpha \in H_0$}, \end{align} 
		since $\log_p(p) = 0$. 
		
		Recall that we have fixed a prime $\p$ of $R$ which is totally split in $K$; in particular, fixing a prime $\p_1$ of $K$ above $\p$, the $2s = |G|$ primes of $K$ dividing $\p$ can be written as 
		$$ \{ \p_1, \ldots, \p_s, \overline{\p_1}, \ldots, \overline{\p_s}\} = 
		   \{ \p_1 = \sigma_1(\p_1), \ldots, \sigma_s(\p_1), \sigma_{s+1}(\p_1), \ldots, \sigma_{2s}(\p_1) \}. $$ 
		Therefore the embeddings $\varphi_i\colon K \longrightarrow K_{\p_i}$ can be written as $\varphi_i = \varphi_1 \circ \sigma_i$ for each $i \in \{1, \ldots, s\}$. 
		
		Now we consider the ideal  
	    $$ I := \left\{ \sum_{i = 1}^{2s} b_i \sigma_i \in R \; \Big| \; \sum_{i = 1}^{2s} b_i \log_p(\varphi_1(\sigma_i(\alpha))) = 0 \textup{ for all $\alpha \in H_0$} \right\} $$ 
	    in the commutative group ring $R = \overline{\Q}_p[G]$. The ideal $I$ contains the elements $$\sigma_i + \sigma_{s+i} = (1 + \tau) \sigma_i$$ 
	    for all $1 \le i \le s$, because $\tau(\alpha) = \alpha^{-1}$ for all $\alpha \in H_0$ (recall that $\tau$ denotes complex conjugation). Therefore $I$ contains $(1 + \tau )R = R^+$. But $I\setminus R^+ \ne \emptyset$ in view of \eqref{extra-gleichung}. 
	
	It follows from \cite[Theorem~33.8]{CuRe62} that $I$ is generated by idempotents $\varepsilon_\chi \in R$, $\chi \in \Ir_p(G)$, which are of the form 
	$$\varepsilon_\chi = \frac{1}{|G|} \cdot \sum_{g \in G} \chi(g) g^{-1}.$$ Since $I$ is strictly larger than $R^+$, we may assume that $\chi(\tau) = - \chi(1)$ for some $\chi$. This means that $I$ contains an element $\sum_{i=1}^{2s} \chi(\sigma_i^{-1}) \sigma_i$ such that 
	\begin{align} \label{eq_contradiction} 
	\sum_{i = 1}^{2s} \chi(\sigma_i^{-1}) \log_p(\varphi_{1}(\sigma_i(\alpha))) = \sum_{i=1}^s 2 \chi(\sigma_i^{-1}) \log_p(\varphi_{1}(\sigma_i(\alpha))) = 0 \quad \text{for all } \alpha \in H_0. 
	\end{align} 
	Let $\tilde{\alpha} \in E_K'$ be divisible by the prime $\p_1$, and by no other prime. We define $\alpha := \tilde{\alpha}/\tau(\tilde{\alpha})$. Then ${\alpha \in H_0}$, and 
	$$\sigma_1(\alpha), \ldots, \sigma_s(\alpha) \in H_0$$ 
	are multiplicatively independent, since the ${\sigma_i(\p_1) = \p_i}$ are just the conjugates of $\p_1$. Therefore $\log_p(\varphi_{1}(\sigma_1(\alpha))), \ldots, \log_p(\varphi_{1}(\sigma_s(\alpha)))$ are linearly independent over $\Q$. By the theorem \cite{Bru67} of Brumer, these logarithms are still linearly independent over some algebraic closure of $\Q$ in $\overline{\Q}_p$. But this contradicts \eqref{eq_contradiction} (note that the values of $\chi$ are algebraic over $\Q$). 
	\end{proof} 
	
	\begin{claim} \label{claim_3} 
		$[(H \cap \mathcal{U}^{p^n}): H^{p^n}]$ remains bounded as $n \to \infty$. 
	\end{claim} 
	\begin{proof} 
		Suppose that the sequence of indices is unbounded. Since the quotient groups $(H \cap \mathcal{U}^{p^n})/H^{p^n}$ have bounded $\Z$-rank ($\le s$), there exists a sequence $(z_n)_n$ of elements in $H$ such that $z_n \in \mathcal{U}^{p^n}$ and such that the order of the coset $\overline{z_n}$ of $z_n$ in ${(H \cap \mathcal{U}^{p^n})/H^{p^n}}$ tends to infinity as $n \to \infty$. We may assume that $\overline{z_{n+1}} \mapsto \overline{z_n}$ under the canonical map 
		$$ (H \cap \mathcal{U}^{p^{n+1}})/H^{p^{n+1}} \twoheadrightarrow (H \cap \mathcal{U}^{p^n})/H^{p^n}. $$ 
		Write $z_n = \varepsilon_1^{b_1^{(n)}} \cdot \ldots \cdot \varepsilon_s^{b_s^{(n)}}$ for suitable elements $b_1^{(n)}, \ldots, b_s^{(n)} \in \Z$ and fixed $\varepsilon_1, \ldots, \varepsilon_s \in H$. Since the order of $\overline{z_n}$ tends to infinity, and as the exponent of the torsion subgroup of $\mathcal{U}$ is finite and $z_n \in \mathcal{U}^{p^n}$ for each $n \in \N$, we may assume that $\varepsilon_1, \ldots, \varepsilon_s \in H$ are $\Z$-linearly independent. 
		
		Since $\overline{z_{n+1}} \mapsto \overline{z_n}$, we have $b_i^{(n+1)} \equiv b_i^{(n)} \pmod{p^n}$ for each $i \in \{1, \ldots, s\}$ and every $n \in \N$, i.e. we can define $a_i := \lim_{n \to \infty} b_i^{(n)} \in \Z_p$, $1 \le i \le s$. Then the element $z := \varepsilon_1^{a_1} \cdot \ldots \cdot \varepsilon_s^{a_s} \in \overline{H}$ is contained in $\bigcap_n \mathcal{U}^{p^n}$, i.e. $z = 1$. Therefore $\rg_{\Z_p}(\overline{H}) < \rg_{\Z}(H)$, yielding a contradiction to the previous Claim \ref{claim_2}. 
	\end{proof} 
	For the last step of the proof, we recall that $A' = \varprojlim_n A_n'$ is a $\Z_p\llbracket\Gamma\rrbracket$-module, and $\Z_p\llbracket\Gamma\rrbracket \cong \Z_p\llbracket T \rrbracket$, where $T$ corresponds to $\gamma - 1$ for some topological generator $\gamma$ of $\Gamma = \Gal(K_\infty/K)$. Note that $\GKC^-(K)$ can be written as follows: 
	$$ \rg_{\Z_p}((A')^-/(T \cdot (A')^-)) = 0. $$ 
	Now we can put the statements of the above claims together in order to prove Theorem~\ref{thm:weihnacht}. The last step is the following 
	\begin{claim} \label{claim_4} 
		In the above situation, we have 
		$$ \rg_{\Z_p}((A')^-/(T \cdot (A')^-)) \le r - s. $$ 
	\end{claim} 
	\begin{proof} 
		Since $\rg_{\Z}(H_0) = s$, it follows from Claim \ref{claim_3} and inequality \eqref{noname} that 
		$$ [H_0: (H_0 \cap N_n(K_n))] \ge \frac{[ H_0: H_0^{p^n}]}{p^C} \ge p^{n \cdot s - C} $$ 
		for each $n \in \N$ and some fixed constant $C$. Since $H_0$ contains only $p$-units $\alpha$ such that $\tau(\alpha) = \alpha^{-1}$, we may conclude that 
		$$ [E_K': (N_n(K_n) \cap E_K')] \ge p^{s \cdot n - C} \cdot [E_{K^+}': (N_n(K_n^+) \cap E_{K^+}')] $$ 
		for every $n \in \N$. 
		
		The Theorem of Chevalley (in the version \eqref{chevalley} used in the proof of Lemma \ref{lemma:chevalley-}) implies that 
		$$ \rg_T((A'_{m+1})^-) \le r - s + \rg_T((A_m')^-) $$ 
		for all sufficiently large $m \in \N$ (recall that $r$ equals the number of primes of $K^+$ dividing $p$ which split in $K$). We may assume that $m$ is large enough such that all the primes of $K_m$ dividing $p$ are totally ramified in $K_\infty$. Therefore we may deduce from \cite[Theorem~2.2]{T-ranks} that 
		$$ \rg_T((A_n')^-) - \rg_T((A_m')^-) \le (n - m )\cdot (r-s)$$ 
		for every $n \ge m$. This implies that $\rg_{\Z_p}((A')^-/(T \cdot (A')^-)) \le r - s$ (cf. also \cite[Remark~2.4]{T-ranks}). 
	\end{proof} 
	This also finishes the proof of Theorem~\ref{thm:weihnacht}. 
\end{proof}

\begin{cor} 
\label{cor_GVC}
  Let $K$ be a finite abelian CM-extension of a totally real number field $R$. Let $\p_1, \ldots, \p_t$ be the primes of $R$ dividing $p$, let $G = \Gal(K/R)$, and write $|G| = n = 2^k \cdot m$ with $k \ge 1$ and $m$ odd. Suppose that \begin{compactenum}[(a)] 
  \item $\p_1$ is totally split in $K/R$ and satisfies $R_{\p_1} = \Q_p$, and that 
  \item the primes of $K^+$ dividing $\p_2, \ldots, \p_t$ are unsplit in $K$. 
  \end{compactenum} Then $\GKC^-(K)$ holds. If moreover $K \cap R_\infty = R$ (e.g. since $p \nmid |G|)$, then $\GVC(K/R, \chi)$ holds for each ${\chi \in \Ir^-(G)}$. More generally, under this assumption we may derive $\GVC(K/\tilde{R}, \chi)$ for every totally real field $\tilde{R}$ satisfying ${R \subseteq \tilde{R} \subseteq K}$ and each $\chi \in \Ir^-(\Gal(K/\tilde{R}))$, respectively. 
\end{cor} 
\begin{proof} 
  Using the notation of Theorem~\ref{thm:weihnacht}, we have $r = s$ by our assumptions. The further statements follow from Corollary~\ref{cor:equivalence} because $\tilde{R}_\chi \subseteq K$ for every $\tilde{R} \subseteq K$ and each $\chi \in \Ir^-(\Gal(K/\tilde{R}))$. 
\end{proof} 

\begin{rem} \label{rem:final} 
  Condition (b) from the above corollary is satisfied, for example, if \begin{compactitem} 
    \item[(i)] the numbers of primes of $K$ above $\p_2, \ldots, \p_t$ are not divisible by 2, or if 
    \item[(ii)] the numbers of primes of $K$ above $\p_2, \ldots, \p_t$ are not divisible by $2^k = \frac{n}{m}$, and the $2$-Sylow subgroup of $G = \Gal(K/R)$ is cyclic (indeed, in this case there exists exactly one element of $G$ of order 2, and therefore our assumption implies that the decomposition fields of $\p_2, \ldots, \p_t$ in $K$ are contained in the unique intermediate field of degree $\frac{n}{2}$ over $R$, which is $K^+$). 
  \end{compactitem} 
\end{rem}

\begin{rem} \label{rem:example} $\;$ \\ \vspace{-3mm} \begin{compactenum}[(1)] 
   \item In the situation of Corollary \ref{cor_GVC}, each character $\chi \in \Ir^-(\Gal(K/R))$ satisfies $r_{S, \chi} = 1$ so these cases can already be derived from the literature.
   \item On the other hand, by using the approach described in Remark \ref{rem:hochgehen}, we can use the above corollary for producing examples of $\GVC(K/\tilde{R}, \chi)$ with large vanishing orders. Indeed, if $\tilde{R}$ is a field as in Corollary~\ref{cor_GVC}, then the $p$-adic $L$-functions attached to characters in $\Ir^-(\Gal(K/\tilde{R}))$ will have vanishing orders at least $[\tilde{R}:R]$, because $\p_1$ is totally split in $K/R$. 
\end{compactenum} 
\end{rem} 

\begin{example}
\label{ex_largevanish}
In order to show that the conditions of Corollary~\ref{cor_GVC} are easily realisable we computed some examples. The totally real subfield $R$ is always non-normal over $\Q$. We define the modulus $\frm_{\infty}$ to be the product of all the infinite places of $R$ and  $\frm_{\mathrm{fin}}$ to be a suitable product of finite primes. Then we set $K=R(\frm_{\infty} \frm_{\mathrm{fin}})$ and we search for moduli $\frm_{\mathrm{fin}}$ such that the hypothesis from Corollary~\ref{cor_GVC} are satisfied for $K$. Below we give a list of some of the examples we found with absolute degree $[K:\Q]$ larger than $10$. Here we identify a monic polynomial ${f(X)= X
^n + \sum_{i=0}^{n-1} a_{i} X^{i} \in \Z[X]}$ with the vector $[a_{0}, a_{1}, \ldots , a_{n-1}]$, and we write $\frp_{q}$ for a prime ideal of $\Ok_{R}$ above $q$. In order to keep the list in a concise form we do not specify which of the prime ideals above $q$ we used. 

In the last column of the table, we write down the degree of the maximal totally real subfield $K^+$ of $K$
over $R$, which marks a lower bound for $r_{S, \chi}$ for the non-trivial character $\chi$ of $\Gal(K/K^+)$, in the sense of Remark~\ref{rem:example},(2). 
\begin{table}[h]
    \centering
    \begin{tabular}{r|r|c|c|c}
    Prime & Def. poly. of $R$ & $\frm_{\mathrm{fin}}$ & Degree of $K$ over $\Q$ & $r_{S, \chi} \ge $ \\
    \hline
    \hline
      $2$ &  $[-12,-26,0]$  &  $\frp_{79}$ & $18$ & 3 \\
      $5$ & $[-13,-20,-1]$ & $\frp_{11}$ & $30$ & 5 \\
      $11$ & $[87,-39,-1]$ & $\frp_{61}$ & $60$ & 10 \\
      \hline
      $7$ & $[7,5,-6,-2]$ & $\frp_{37}$ & $16$ & 2 \\
      $23$ & $[4,8,-9,-2]$ & $\frp_{31}$ & $24$ & 3
    \end{tabular}
    \label{tab_ex_cor} 
\end{table}
\end{example} 

Finally we produce an example where $K/R$ is non-abelian. Although the corresponding character $\chi$ satisfies $r_{S, \chi} = 1$ and therefore we do not obtain new, previously unknown cases, this theorem nicely illustrates possible further applications of the rank inequality from Theorem~\ref{thm:weihnacht}.  
\begin{thm} \label{lemma:final} 
  Let $K$ be a normal extension of a totally real number field $R$. Let $\p_1, \ldots, \p_t$ be the primes of $R$ dividing $p$. We assume that \begin{compactenum}[(a)] 
    \item $\p_1$ is totally split in $K$ and $R_{\p_1} = \Q_p$, 
    \item the primes of $K^+$ dividing $\p_2, \ldots, \p_t$ do not split in $K$, 
    \item $G := \Gal(K/R)$ is isomorphic to the dihedral group $D_{n}$ with $2n$ elements, where $n \equiv 2 \pmod{4}$, and 
  \end{compactenum} 
  Then $\GKC(K/R, \chi)$ holds for some $\chi \in \Ir^-(G)$. 
\end{thm} 
  \begin{proof} 
    Write $G = D_n$ as 
    $$ G = \langle a, b \mid a^n = b^2 = 1, b a b = a^{-1} \rangle. $$ 
    Then the complex conjugation $\tau$ corresponds to the only non-trivial element $a^{n/2}$ in the center of $G$. We let $\tilde{R}$ be the subfield of $K$ which is fixed by the subgroup $\langle a \rangle$ generated by $a$. Then $\tilde{R}$ is a degree 2 extension of $R$, and it is totally real because $\tilde{R} \subseteq K^+ = K^{\langle \tau \rangle}$. It follows from Theorem \ref{thm:weihnacht}, applied to the extension $K/\tilde{R}$, that 
    $$ \rg_{\Z_p}((A')^-/(T \cdot (A')^-)) \le r - s, $$ 
    where $r$ denotes the number of primes of $K^+$ above $p$ which split in $K$, i.e. 
    $$r = \frac{[K:R]}{2} = n$$ by our assumptions, and where $s := \frac{[K : \tilde{R}]}{2} = \frac{n}{2}$. Now the representation theory of the dihedral group $D_{n}$ implies that $\Ir^-(G)$ contains $\frac{n-2}{4}$ characters of dimension 2 and two characters of dimension 1. 
    
    Since the proof of Theorem~A, and in particular equation \eqref{eq:T-ord}, implies that $\ord_{T = 0} f_{A', \chi}(T)$ is divisible by $\chi(1)$, and as 
    $$ \sum_{\chi \in \Ir^-(G)} \chi(1) = \frac{n-2}{4} \cdot 2 + 2 \cdot 1 = \frac{n}{2} + 1 > \frac{n}{2} = r - s$$ 
    by the above, it follows that $\ord_{T=0} f_{A', \chi}(T) = 0$ for at least one $\chi \in \Ir^-(G)$. 
  \end{proof} 

\bibliography{gross} 

\bibliographystyle{alpha}

\textsc{Ludwig-Maximilians-Universit\"at M\"unchen,
Mathematisches Institut,
Theresienstr. 39, 80333 M\"unchen,
Germany}\\
\textit{Email address:} \href{mailto:hofer@math.lmu.de}{hofer@math.lmu.de}

\hspace{1cm}

\textsc{Universit\"at der Bundeswehr M\"unchen, Institut f\"ur theoretische Informatik, Mathematik und Operations Research, 85577 Neubiberg, Germany} \\
\textit{Email address:} \href{mailto:soeren.kleine@unibw.de}{soeren.kleine@unibw.de}

\end{document}